\documentclass[12pt,oneside,reqno]{amsart}

\usepackage{amsmath,amssymb,amsthm,textcomp}
\usepackage{amsfonts,graphicx}
\usepackage[mathscr]{eucal}
\usepackage{color}
\usepackage{csquotes}
\usepackage[backend=bibtex,%
firstinits=true,%
doi=false,%
isbn=true,%
url=false,%
maxnames=99]{biblatex}%

\AtEveryBibitem{\clearfield{issn}}
\AtEveryCitekey{\clearfield{issn}}
\numberwithin{equation}{section}
\DeclareNameAlias{sortname}{last-first}
\theoremstyle{definition}

\numberwithin{equation}{section}

\newcommand{\ncom}{\newcommand}

\ncom{\beq}{\begin{equation}}
\ncom{\eeq}{\end{equation}}
\ncom{\bea}{\begin{eqnarray*}}
\ncom{\eea}{\end{eqnarray*}}
\ncom{\beqa}{\begin{eqnarray}}
\ncom{\eeqa}{\end{eqnarray}}
\ncom{\nno}{\nonumber}
\ncom{\non}{\nonumber}
\ncom{\ds}{\displaystyle}
\ncom{\half}{\frac{1}{2}}
\ncom{\mbx}{\makebox{.25cm}}
\ncom{\hs}{\mbox{\hspace{.25cm}}}
\ncom{\rar}{\rightarrow}
\ncom{\Rar}{\Rightarrow}
\ncom{\noin}{\noindent}
\ncom{\bc}{\begin{center}}
\ncom{\ec}{\end{center}}
\ncom{\sz}{\scriptsize}
\ncom{\rf}{\ref}
\ncom{\s}{\sqrt{2}}
\ncom{\sgm}{\sigma}
\ncom{\Sgm}{\Sigma}
\ncom{\psgm}{\sigma^{\prime}}
\ncom{\dt}{\delta}
\ncom{\Dt}{\Delta}
\ncom{\lmd}{\lambda}
\ncom{\Lmd}{\Lambda}
\ncom{\Th}{\Theta}
\ncom{\e}{\eta}
\ncom{\eps}{\epsilon}
\ncom{\pcc}{\stackrel{P}{>}}
\ncom{\lp}{\stackrel{L_{p}}{>}}
\ncom{\dist}{{\rm\,dist}}
\ncom{\sspan}{{\rm\,span}}
\ncom{\re}{{\rm Re\,}}
\ncom{\im}{{\rm Im\,}}
\ncom{\sgn}{{\rm sgn\,}}
\ncom{\ba}{\begin{array}}
\ncom{\ea}{\end{array}}
\ncom{\hone}{\mbox{\hspace{1em}}}
\ncom{\htwo}{\mbox{\hspace{2em}}}
\ncom{\hthree}{\mbox{\hspace{3em}}}
\ncom{\hfour}{\mbox{\hspace{4em}}}
\ncom{\vone}{\vskip 2ex}
\ncom{\vtwo}{\vskip 4ex}
\ncom{\vonee}{\vskip 1.5ex}
\ncom{\vthree}{\vskip 6ex}
\ncom{\vfour}{\vspace*{8ex}}
\ncom{\norm}{\|\;\;\|}
\ncom{\integ}[4]{\int_{#1}^{#2}\,{#3}\,d{#4}}
\ncom{\vspan}[1]{{{\rm\,span}\{ #1 \}}}
\ncom{\dm}[1]{ {\displaystyle{#1} } }
\ncom{\ri}[1]{{#1} \index{#1}}

\newtheorem{theorem}{\bf Theorem}[section]

\newtheorem{lemma}{Lemma}[section]
\newtheorem{corollary}{Corollary}[section]

\newtheoremstyle
    {remarkstyle}
    {}
    {11pt}
    {}
    {}
    {\bfseries}
    {:}
    {     }
    {\thmname{#1} \thmnumber{#2} }

\theoremstyle{remarkstyle}

\def\eps{\varepsilon}

\begin{document}
\title{Correlation between Adomian and partial exponential Bell polynomials}
\author[Kuldeep Kumar Kataria]{K. K. Kataria}
\address{Kuldeep Kumar Kataria, Department of Mathematics,
 Indian Institute of Technology Bombay, Powai, Mumbai 400076, INDIA.}
 \email{kulkat@math.iitb.ac.in}
\author{P. Vellaisamy}
\address{P. Vellaisamy, Department of Mathematics,
 Indian Institute of Technology Bombay, Powai, Mumbai 400076, INDIA.}
 \email{pv@math.iitb.ac.in}
\thanks{The research of K. K. Kataria was supported by a UGC fellowship no. F.2-2/98(SA-I), Govt. of India.}
\begin{abstract}
We obtain some recurrence relationships among the partition vectors of the partial exponential Bell polynomials. On using such results, the $n$-th Adomian polynomial for any nonlinear operator can be expressed explicitly in terms of the partial exponential Bell polynomials. Some new identities for the partial exponential Bell polynomials are obtained by solving certain ordinary differential equations using Adomian decomposition method.
\end{abstract}

\maketitle
\section{Introduction}
The Bell polynomials studied by {\small{\sc Bell}} (1927), (1934) are special polynomials in combinatorial analysis with numerous applications in different areas of mathematics. The incomplete or partial exponential Bell polynomials $B_{n,k}$ (see p. 96, {\small{\sc Hazewinkel}} (1997), {\small{\sc Cvijovi\'c}} (2011)) in $n-k+1$ variables  are triangular array of polynomials defined by
\begin{equation}\label{1.10}
B_{n,k}(u_1,u_2,\ldots,u_{n-k+1})=n!\underset{\Lambda^k_n}{\sum}\prod_{j=1}^{n-k+1}\frac{1}{k_j!}\left(\frac{u_j}{j!}\right)^{k_j},
\end{equation}
where the partition set is given by
\begin{equation*}\label{1.11}
\Lambda^k_n=\left\{\left(k_1,k_2,\ldots,k_{n-k+1}\right):\sum_{j=1}^{n-k+1}k_j=k, \sum_{j=1}^{n-k+1}jk_j=n, k_j\in\mathbb{N}_0\right\}.
\end{equation*}
Here $\mathbb{N}_m=\left\{x:x\geq m,x\in\mathbb{N}\cup\{0\}\right\}$ and $\mathbb{N}$ denotes the set of positive integers. Also, the sum
\begin{equation}
B_{n}(u_1,u_2,\ldots,u_{n})=\sum_{k=1}^{n‎}B_{n,k}(u_1,u_2,\ldots,u_{n-k+1}),
\end{equation}
is called the $n$-th complete exponential Bell polynomial. The partial ordinary Bell polynomials in $n-k+1$ variables are defined by
\begin{equation}\label{nn1.10}
\hat{B}_{n,k}(u_1,u_2,\ldots,u_{n-k+1})=k!\underset{\Lambda^k_n}{\sum}\prod_{j=1}^{n-k+1}\frac{u_j^{k_j}}{k_j!}.
\end{equation}
For more details, results and some known identities on Bell polynomials, we refer to pp. 133-137, {\small{\sc Comtet}} (1974), pp. 95-98, {\small{\sc Hazewinkel}} (1997) and {\small{\sc Cvijovi\'c}} (2011).

Next we briefly explain Adomian decomposition method (ADM) (see {\small{\sc Adomian}} (1986), (1994)), which will be used latter to obtain some new identities for Bell polynomials.

In ADM, solution of the functional equation
\begin{equation}\label{1.1}
u=f+L(u)+N(u),
\end{equation}
where $L$ and $N$ are linear and nonlinear operators respectively and $f$ is a known function, is expressed in the form of an infinite series
\begin{equation}\label{1.2}
u= ‎‎\sum_{n=0}^{\infty‎}‎u_n.
\end{equation}
The nonlinear term $N(u)$ decomposes as 
\begin{equation}\label{1.3}
N(u)=‎‎\sum_{n=0}^{\infty‎}A_n(u_0,u_1,\ldots,u_n),
\end{equation}
where $A_n$ denotes the $n$-th Adomian polynomial in $u_0,u_1,\ldots,u_n$. Also, the series (\ref{1.2}) and (\ref{1.3}) are assumed to be absolutely convergent. So, (\ref{1.1}) can be rewritten as
\begin{equation*}\label{1.4}
‎‎\sum_{n=0}^{\infty‎}‎u_n=f+‎‎\sum_{n=0}^{\infty‎}L(u_n)+‎‎\sum_{n=0}^{\infty‎}A_n.
\end{equation*}
Thus $u_n$'s are obtained by the following recursive relation
\begin{equation*}\label{1.5}
u_0=f\ \ \ \ \mathrm{and}\ \ \ \ u_n=L(u_{n-1})+A_{n-1}.
\end{equation*}
The crucial step involved in ADM is the calculation of Adomian polynomials. {\small{\sc Adomian}} (1986) gave a method for determining these polynomials, by parametrizing $u$ as
\begin{equation*}\label{1.4a}
u_\lambda=‎‎\sum_{n=0}^{\infty‎}u_n\lambda^n 
\end{equation*}
and assuming $N(u_\lambda)$ to be analytic in $\lambda$, which decomposes as
\begin{equation*}\label{1.5a}
N(u_\lambda)=‎‎\sum_{n=0}^{\infty‎}A_n(u_0,u_1,\ldots,u_n)\lambda^n.
\end{equation*}
Hence, Adomian polynomials are given by
\begin{equation}\label{1.6}
A_n(u_0,u_1,\ldots,u_n)=\left.\frac{1}{n!}\frac{\partial^n N(u_\lambda)}{\partial \lambda^n} \right|_{\lambda=0},\ \forall\ n\in\mathbb{N}_0.
\end{equation}

An improved version of the above result (see {\small{\sc Zhu}} {\it et al.} (2005)) is given by
\begin{equation}\label{1.678}
A_n(u_0,u_1,\ldots,u_n)=\left.\frac{1}{n!}\frac{\partial^n N(‎‎\sum_{k=0}^{n‎}‎u_k\lambda^k)}{\partial \lambda^n} \right|_{\lambda=0},\ \forall\ n\in\mathbb{N}_0.
\end{equation}

{\small{\sc Rach}} (1984) suggested the following formula for these polynomials:
\begin{equation*}\label{1.7a}
A_0(u_0)=‎‎N(u_0),
\end{equation*}
and
\begin{equation}\label{1.7}
A_n(u_0,u_1,\ldots,u_n)=\sum_{k=1}^{n‎}C(k,n)N^{(k)}(u_0),\ \forall\ n\in\mathbb{N},
\end{equation}
where 
\begin{equation}\label{1.8}
C(k,n)=\underset{\Theta^k_n}{\sum}\prod_{j=1}^{n}\frac{u_j^{k_j}}{k_j!},
\end{equation}
and the summation is taken over the partition set
\begin{equation*}\label{1.9}
\Theta^k_n=\left\{\left(k_1,k_2,\ldots,k_n\right):\sum_{j=1}^nk_j=k, \sum_{j=1}^{n}jk_j=n, k_j\in\mathbb{N}_0\right\}.
\end{equation*}
Also, $N^{(k)}(.)$ denotes the $k$-th derivative of the nonlinear term. One can easily show the equivalence of (\ref{1.6}) and (\ref{1.7}) using the Fa\`{a} di Bruno's formula. Recently, {\small{\sc Kataria}} and {\small{\sc Vellaisamy}} (2016) obtained simple parametrization methods for generating these Adomian polynomials both explicitly and recursively.

We obtain some results related to the partition vectors of the partial exponential Bell polynomials. An important observation is that the $C(k,n)$'s are homogeneous polynomials of order $k$ which can be represented in terms of the well known Bell polynomials. Hence, a closed form expression of the $n$-th order Adomian polynomial for any nonlinear operator is obtained as a finite sum of the partial exponential Bell polynomials. The significance is that any algorithm or an identity for $C(k,n)$'s will give the corresponding results for the Bell polynomials. Indeed we use the results of {\small{\sc Duan}} (2010), (2011) to obtain some recursive algorithms for the partial exponential Bell polynomials. Also, we use Adomian decomposition method to solve certain ordinary differential equations to obtain some new identities for the partial exponential Bell polynomials. The corresponding results for the partial ordinary Bell polynomials are mentioned as corollaries.

\section{The partition vectors of Bell polynomials}
We use the following results by {\small{\sc Duan}} (2010) for the partition set $\Theta^k_n$ to show some similar results for the partition set $\Lambda^k_n$.
\begin{lemma}\label{l1a}
	For $1\leq k\leq n$, $\Theta^k_n\subset\mathbb{N}_0^n$ and $\Theta^1_1=\{(1)\}$, $\Theta^1_n=\{(0,0,\ldots,0,1)\}$.
\end{lemma}
\begin{lemma}\label{l2}
	For every vector $\left(k_1,k_2,\ldots,k_{n}\right)\in\Theta^k_n$, $2\leq k\leq n$, the last $(k-1)$ entries are zero \textit{i.e.} $k_{n-k+2}=k_{n-k+3}=\ldots=k_n=0$.
\end{lemma}
\begin{theorem}\label{t1a}
	For $n\in\mathbb{N}_2$, if $2\leq k\leq\left[\frac{n}{2}\right]$, then $\Theta^k_n=\Theta_1\cup\Theta_2$ and if $\left[\frac{n}{2}\right]<k\leq n$, then $\Theta^k_n=\Theta_1$, where
	\begin{equation*}
	\Theta_1=\{\left(k_1+1,k_2,\ldots,k_{n-1},0\right):\left(k_1,k_2,\ldots,k_{n-1}\right)\in\Theta^{k-1}_{n-1}\}
	\end{equation*}
	and
	\begin{equation*}
	\Theta_2=\{(0,k_1,k_2,\ldots,k_{n-k},\underset{k-1\ \mathrm{times}}{\underbrace{0,0,\ldots,0}}):\left(k_1,k_2,\ldots,k_{n-k}\right)\in\Theta^{k}_{n-k}\}.
	\end{equation*}
\end{theorem}

Next we obtain some results for the partition set $\Lambda^k_n$ of Bell polynomials. The following lemmas are easy to prove.
\begin{lemma}\label{l1}
	For $1\leq k\leq n$, $\Lambda^k_n\subset\mathbb{N}_0^{n-k+1}$ with $\Lambda^1_n=\Theta^1_n$ for all $n\in\mathbb{N}$.
\end{lemma}
The following lemma is evident from Lemma \ref{l2}.
\begin{lemma}\label{l3}
	Let $e^n_j$ denotes the $n$-tuple vector with unity at the $j$-th place and zero elsewhere. Then
	\begin{equation*}\label{2.0}
	\Theta^k_n=\left\{\sum_{j=1}^{n-k+1}k_je^n_j:\left(k_1,k_2,\ldots,k_{n-k+1}\right)\in\Lambda^k_n\right\}.
	\end{equation*}
\end{lemma}
\begin{lemma}\label{l4}
	The summation involved in (\ref{1.10}) and (\ref{1.8}) is invariant under $\Lambda^k_n$ and $\Theta^k_n$.
\end{lemma}
Similar recurrence relationships among the partition vectors of the partial exponential Bell polynomials hold.
\begin{theorem}\label{t2a}
	For $n\in\mathbb{N}_2$, if $2\leq k\leq\left[\frac{n}{2}\right]$, then $\Lambda^k_n=\Lambda_1\cup\Lambda_2$ and if $\left[\frac{n}{2}\right]<k\leq n$, then $\Lambda^k_n=\Lambda_1$, where
	\begin{equation*}
	\Lambda_1=\{\left(k_1+1,k_2,\ldots,k_{n-k+1}\right):\left(k_1,k_2,\ldots,k_{n-k+1}\right)\in\Lambda^{k-1}_{n-1}\}
	\end{equation*}
	and
	\begin{equation*}
	\Lambda_2=\{(0,k_1,k_2,\ldots,k_{n-2k+1},\underset{k-1\ \mathrm{times}}{\underbrace{0,0,\ldots,0}}):\left(k_1,k_2,\ldots,k_{n-2k+1}\right)\in\Lambda^{k}_{n-k}\}.
	\end{equation*}
\end{theorem}
\begin{proof}
	The proof is evident on using Lemma \ref{l2} and Theorem \ref{t1a}.
\end{proof}
\begin{table}[ht]\label{table}
	\caption{The $6\times 6$ array of the partition vectors for Bell polynomials.}
	\begin{tabular}{ccccccc}
		\hline
		$n\setminus k$\quad\quad\quad& $1$  \quad& $2$\quad& $3$ \quad& $4$ \quad& $5$ \quad& $6$\\
		\hline\\
		$1$\quad& $(1)$ \quad& \quad& \quad& \quad& \quad&\\
		
		$2$\quad& $(0,1)$ \quad& $(2)$ \quad& \quad& \quad& \quad&\\
		
		$3$\quad& $(0,0,1)$ \quad& $(1,1)$ \quad& $(3)$\quad& \quad& \quad&\\
		
		$4$\quad& $(0,0,0,1)$ \quad& $(1,0,1)$ \quad& $(2,1)$\quad& $(4)$\quad& \quad&\\
		
		\quad&\quad& $(0,2,0)$ \quad& \quad& \quad& \quad&\\
		
		$5$\quad& $(0,0,0,0,1)$ \quad& $(1,0,0,1)$ \quad& $(2,0,1)$\quad& $(3,1)$\quad& $(5)$\quad&\\
		
		\quad&\quad& $(0,1,1,0)$ \quad& $(1,2,0)$\quad& \quad& \quad&\\
		
		$6$\quad& $(0,0,0,0,0,1)$ \quad& $(1,0,0,0,1)$ \quad& $(2,0,0,1)$\quad& $(3,0,1)$\quad& $(4,1)$\quad&$(6)$\\
		
		\quad&\quad& $(0,1,0,1,0)$ \quad& $(1,1,1,0)$\quad& $(2,2,0)$\quad& \quad&\\
		
		\quad&\quad& $(0,0,2,0,0)$ \quad& $(0,3,0,0)$\quad& \quad& \quad&\\
		\hline
	\end{tabular}
\end{table}
For $1\leq n\leq 6$ the vectors for the set $\Lambda^{k}_{n}$ is listed in Table 1. 

A relationship between Adomian polynomials and Bell polynomials was first obtained by {\small{\sc Abbaoui}} {\it et al.} (1995). Here we establish a relationship between Adomian and partial exponential Bell polynomials, which is different from the one proved in {\small{\sc Abbaoui}} {\it et al.} (1995).
\begin{theorem}\label{t1}
	Let $A_n$, $n\geq 1$, be the $n$-th Adomian polynomial for the nonlinear term $N(u)$. Then
	\begin{equation}\label{2.1we}
	A_n(u_0,u_1,\ldots,u_n)=\frac{1}{n!}\sum_{k=1}^{n‎}B_{n,k}(1!u_1,2!u_2,\ldots,(n-k+1)!u_{n-k+1})N^{(k)}(u_0).
	\end{equation}
\end{theorem}
\begin{proof}
	On setting $k_{n-k+j}=0$ for $j=2,3,\ldots,k$ and using Lemma \ref{l4}, we have
	\begin{eqnarray*}
		B_{n,k}(1!u_1,2!u_2,\ldots,(n-k+1)!u_{n-k+1})&=&n!\underset{\sum_{j=1}^{n-k+1}jk_j=n\ ,\ k_j\in\mathbb{N}_0}{\sum_{\sum_{j=1}^{n-k+1}k_j=k}}\prod_{j=1}^{n-k+1}\frac{u_j^{k_j}}{k_j!}\nonumber\\
		&=&n!\underset{\sum_{j=1}^{n}jk_j=n\ ,\ k_j\in\mathbb{N}_0}{\sum_{\sum_{j=1}^{n}k_j=k}}\prod_{j=1}^{n}\frac{u_j^{k_j}}{k_j!}\nonumber\\
		&=&n!C(k,n).
	\end{eqnarray*}
	Hence,
	\begin{equation}\label{2.2nb}
	C(k,n)=\frac{1}{n!}B_{n,k}(1!u_1,2!u_2,\ldots,(n-k+1)!u_{n-k+1}).
	\end{equation}
	The proof completes on using (\ref{1.7}).
\end{proof}
From (\ref{2.2nb}) it is clear that the homogeneous polynomials of order $k$ \textit{i.e.} $C(k,n)$'s are the partial exponential Bell polynomials. In view of (\ref{1.10}) and (\ref{nn1.10}), the partial exponential Bell polynomial is connected to the partial ordinary Bell polynomial as follows:
\begin{equation}\label{ew1.10}
\hat{B}_{n,k}(u_1,u_2,\ldots,u_{n-k+1})=\frac{k!}{n!}B_{n,k}(1!u_1,2!u_2,\ldots,(n-k+1)!u_{n-k+1}).
\end{equation}
\begin{corollary}\label{cdt1}
	The $n$-th Adomian polynomial for $N(u)$ is given by
	\begin{equation*}\label{2.1}
	A_n(u_0,u_1,\ldots,u_n)=\sum_{k=1}^{n‎}\frac{1}{k!}\hat{B}_{n,k}(u_1,u_2,\ldots,u_{n-k+1})N^{(k)}(u_0).
	\end{equation*}
\end{corollary}
On using (\ref{2.2nb}), the following results follows from the recursive algorithm of $C(k,n)$ given by Corollary 1 of {\small{\sc Duan}} (2010).
\begin{corollary}\label{fgh3}
	For all $n\in\mathbb{N}$, we have $B_{n,1}(1!u_1,2!u_2,\ldots,n!u_n)=n!u_n$ and $B_{n,n}(u_1)=u_1^n$. When $\left[\frac{n}{2}\right]<k\leq n$ and $n\in\mathbb{N}_2$, we have
	\begin{align*}
	B_{n,k}(1!u_1,2!u_2,\ldots,&(n-k+1)!u_{n-k+1})\\
	&=n\left.B_{n-1,k-1}(1!u_1,2!u_2,\ldots,(n-k+1)!u_{n-k+1})\right|_{k_1\rightarrow k_1+1}.
	\end{align*}
	For $2\leq k\leq \left[\frac{n}{2}\right]$ and $n\in\mathbb{N}_4$, we have
	\begin{align*}
	B_{n,k}(1!u_1,2!u_2,\ldots,&(n-k+1)!u_{n-k+1})\\
	&=n\left.B_{n-1,k-1}(1!u_1,2!u_2,\ldots,(n-k+1)!u_{n-k+1})\right|_{k_1\rightarrow k_1+1}\\
	&\ \ \ +(n)_kB_{n-k,k}(1!u_2,2!u_3,\ldots,(n-2k+1)!u_{n-2k+2}),
	\end{align*}
	where $(n)_k=n(n-1)\ldots (n-k+1)$, denotes the falling factorials.
\end{corollary}
\begin{corollary}\label{fgh34}
	For all $n\in\mathbb{N}$, we have $\hat{B}_{n,1}(u_1,u_2,\ldots,u_n)=u_n$ and $\hat{B}_{n,n}(u_1)=u_1^n$. For $n\in\mathbb{N}_2$ and $\left[\frac{n}{2}\right]<k\leq n$, we have
	\begin{equation*}
	\hat{B}_{n,k}(u_1,u_2,\ldots,u_{n-k+1})=k\left.\hat{B}_{n-1,k-1}(u_1,u_2,\ldots,u_{n-k+1})\right|_{k_1\rightarrow k_1+1}.
	\end{equation*}
	When $2\leq k\leq \left[\frac{n}{2}\right]$ and $n\in\mathbb{N}_4$, we have
	\begin{align*}
	\hat{B}_{n,k}(u_1,u_2,\ldots,u_{n-k+1})&=k\left.\hat{B}_{n-1,k-1}(u_1,u_2,\ldots,u_{n-k+1})\right|_{k_1\rightarrow k_1+1}\\
	&\ \ \ +\hat{B}_{n-k,k}(u_2,u_3,\ldots,u_{n-2k+2}).
	\end{align*}
\end{corollary}
An alternate recursive algorithm for the Bell polynomials follow from Corollary 3 and 4 of {\small{\sc Duan}} (2010).
\begin{corollary}\label{fgh37}
	Let $n$ be any positive integer. Then for $2\leq k\leq n$, we have
	\begin{align*}
	B&_{n,k}(1!u_1,2!u_2,\ldots,(n-k+1)!u_{n-k+1})\\
	&=\sum_{j=0}^{n-k}(j+1)(n-1)_ju_{j+1}B_{n-j-1,k-1}(1!u_1,2!u_2,\ldots,(n-j-k+1)!u_{n-j-k+1}).
	\end{align*}
	Alternatively, for $2\leq k\leq n-1$, we have
	\begin{align*}
	B_{n,k}(1!u_1,2!u_2,\ldots,&(n-k+1)!u_{n-k+1})\\
	&=u_1B_{n-1,k-1}(1!u_1,2!u_2,\ldots,(n-k+1)!u_{n-k+1})\\
	&\ \ \ +\sum_{j=1}^{n-k}(j+1)u_{j+1}\frac{\partial}{\partial u_j}B_{n-1,k}(1!u_1,2!u_2,\ldots,(n-k)!u_{n-k}).
	\end{align*}
\end{corollary}
\begin{corollary}\label{fgh37}
	Let $n$ be any positive integer. Then for $2\leq k\leq n$, we have
	\begin{equation*}
	\hat{B}_{n,k}(u_1,u_2,\ldots,u_{n-k+1})=\frac{k}{n}\sum_{j=0}^{n-k}(j+1)u_{j+1}\hat{B}_{n-j-1,k-1}(u_1,u_2,\ldots,u_{n-j-k+1}).
	\end{equation*}
	Alternatively, for $2\leq k\leq n-1$, we have
	\begin{align*}
	\hat{B}_{n,k}(u_1,u_2,\ldots,u_{n-k+1})&=\frac{k}{n}u_1\hat{B}_{n-1,k-1}(u_1,u_2,\ldots,u_{n-k+1})\\
	&\ \ \ +\frac{1}{n}\sum_{j=1}^{n-k}(j+1)u_{j+1}\frac{\partial}{\partial u_j}\hat{B}_{n-1,k}(u_1,u_2,\ldots,u_{n-k}).
	\end{align*}
\end{corollary}
\noindent {\small{\sc Remark 1.}}
For the linear case $N(u)=u$, we have $A_0(u_0)=u_0$ and
\begin{equation*}\label{2.1kk}
A_n(u_0,u_1,\ldots,u_n)=\frac{1}{n!}B_{n,1}(1!u_1,2!u_2,\ldots,n!u_{n})=u_n,\ \ n\in\mathbb{N}.
\end{equation*}
\noindent {\small{\sc Example 1.}} The Adomian polynomials for $N(u)= e^u$ are $A_0(u_0)= e^{u_0}$ and
\begin{eqnarray*}\label{2.3} 
	A_1(u_0,u_1)&=&B_{1,1}(1!u_1)e^{u_0}=u_1e^{u_0},\\
	A_2(u_0,u_1,u_2)&=&\frac{e^{u_0}}{2!}\left[B_{2,1}(1!u_1,2!u_2)+B_{2,2}(1!u_1)\right]=\left(u_2 + \frac{u_1^2}{2}\right)e^{u_0},\\
	A_3(u_0,u_1,u_2,u_3)&=&\frac{e^{u_0}}{3!}\left[B_{3,1}(1!u_1,2!u_2,3!u_3)+B_{3,2}(1!u_1,2!u_2)+B_{3,3}(1!u_1)\right]\\
	&=&\left(u_3 + u_1u_2+\frac{u_1^3}{6}\right)e^{u_0}.\\
	&\vdots&\\
	A_n(u_0,u_1,\ldots,u_n)&=&\frac{e^{u_0}}{n!}\sum_{k=1}^{n‎}B_{n,k}(1!u_1,2!u_2,\ldots,(n-k+1)!u_{n-k+1})\\
	&=&e^{u_0}\sum_{k=1}^{n‎}\frac{1}{k!}\hat{B}_{n,k}(u_1,u_2,\ldots,u_{n-k+1}).
\end{eqnarray*}
Therefore the $n$-th Adomian polynomial for exponential non-linearity can be expressed in terms of the $n$-th complete exponential Bell polynomial as follows:
\begin{equation*}
A_n(u_0,u_1,\ldots,u_n)=\frac{e^{u_0}}{n!}B_{n}(1!u_1,2!u_2,\ldots,n!u_{n}).
\end{equation*}

\section{Some identities for Bell polynomials}
We now state and prove some new identities for the partial exponential Bell polynomials. These are achieved by solving certain ordinary differential equations using two different methods, namely, the separation of variable method and Adomian decomposition method. The corresponding identities for the partial ordinary Bell polynomials are given as corollaries.
\begin{theorem}\label{t2kk}
	Let $\alpha,\beta$ be any real numbers and $n$ be any positive integer. Then
	\begin{equation}\label{2.4kk}
	\sum_{k=1}^{n‎}\left(-\beta\right)^kB_{n,k}(0!\alpha,-1!\alpha^2\beta,\ldots,\left(-1\right)^{n-k}(n-k)!\alpha^{n-k+1}\beta^{n-k})=n!(-\alpha\beta)^n.
	\end{equation}
\end{theorem}
\begin{proof}
	Consider the following ordinary differential equation
	\begin{equation}\label{2.5kk}
	\frac{du}{dx}=\alpha e^{-\beta u},\ \ \ \ u(0)=1,\ \ \ |\alpha\beta x|<e^\beta.
	\end{equation}
	Equivalently,
	\begin{equation}\label{2.6kk}
	u(x)=u(0)+\alpha \int_0^{x}e^{-\beta u(t)}\,\mathrm{d}t.
	\end{equation}
	In (\ref{2.6kk}), the nonlinear term is $N(u)=e^{-\beta u}$. Substituting $u=\sum_{n=0}^{\infty}u_n$ and $N(u)=\sum_{n=0}^{\infty}A_n$ in the above equation and applying ADM, we get
	\begin{equation*}\label{2.7kk}
	\sum_{n=0}^{\infty}u_n(x)=u(0)+\alpha\sum_{n=0}^{\infty} \int_0^{x}A_n\left(u_0(t),u_1(t),\ldots,u_n(t)\right)\,\mathrm{d}t.
	\end{equation*}
	Therefore, for all $n\geq 0$,
	\begin{equation}\label{2.8kk}
	u_{n+1}(x)=\alpha \int_0^{x}A_{n}\left(u_0(t),u_1(t),\ldots,u_n(t)\right)\,\mathrm{d}t.
	\end{equation}
	On comparing $u_0=u(0)=1$ and hence $u_1,u_2,u_3,u_4$ are recursively obtained as
	\begin{eqnarray*}
		A_0\left(u_0(x)\right)&=&e^{-\beta u_0}=\frac{1}{e^{\beta}},\\
		u_1(x)&=&\alpha\int_0^{x}A_0\left(u_0(t)\right)\,\mathrm{d}t=\frac{\alpha}{e^{\beta}}x,\\
		A_1\left(u_0(x),u_1(x)\right)&=&-\beta u_1e^{-\beta u_0}=-\frac{\alpha\beta}{e^{2\beta}} x,\\
		u_2(x)&=&\alpha\int_0^{x}A_1\left(u_0(t),u_1(t)\right)\,\mathrm{d}t=-\frac{\alpha^2\beta}{2e^{2\beta}} x^2,\\
		A_2\left(u_0(x),u_1(x),u_2(x)\right)&=&\frac{1}{2!}\left(\beta^2u_1^2-2\beta u_2\right)e^{-\beta u_0}=\frac{\alpha^2\beta^2}{e^{3\beta}} x^2,\\
		u_3(x)&=&\alpha\int_0^{x}A_2\left(u_0(t),u_1(t),u_2(t)\right)\,\mathrm{d}t=\frac{\alpha^3\beta^2}{3e^{3\beta}} x^3,\\
		A_3\left(u_0(x),u_1(x),u_2(x),u_3(x)\right)&=&\frac{1}{3!}\left(-\beta^3u_1^3+6\beta^2u_1u_2-6\beta u_3\right)e^{-\beta u_0}\\
		&=&-\frac{\alpha^3\beta^3}{e^{4\beta}} x^3,\\
		u_4(x)&=&\alpha\int_0^{x}A_3\left(u_0(t),u_1(t),u_2(t),u_3(t)\right)\,\mathrm{d}t\\
		&=&-\frac{\alpha^4\beta^3}{4e^{4\beta}} x^4,
	\end{eqnarray*}
	and so on. By using the separation of variables method, it is easy to see that $u(x)=1+\beta^{-1}\ln (1+\alpha\beta e^{-\beta} x)$ is the solution of (\ref{2.5kk}). Since $|\alpha\beta e^{-\beta}x|<1$, by Taylor series expansion of $\ln (1+\alpha\beta e^{-\beta} x)$, we have
	\begin{equation}\label{2.9kk}
	u(x)=1+\sum_{n=1}^{\infty}(-1)^{n+1}\frac{\alpha^n\beta^{n-1}}{ne^{\beta n}}x^n=u(0)+\sum_{n=1}^{\infty}u_n(x).
	\end{equation}
	By uniqueness of the solution of the differential equation (\ref{2.5kk}), the series solution obtained by ADM is consistent with (\ref{2.9kk}). Therefore, the $n$-th component of $u(x)$ is
	\begin{equation}\label{2.10kk}
	u_{n}(x)=(-1)^{n+1}\frac{\alpha^n\beta^{n-1}}{ne^{\beta n}}x^{n},\ \ n\geq1.
	\end{equation}
	Now using Rach formula (\ref{1.7}) in (\ref{2.8kk}), we obtain
	\begin{eqnarray*}\label{2.11kk}
		u_{n+1}(x)&=&\alpha \int_0^{x}\sum_{k=1}^{n}N^{(k)}(u_0)\underset{\Theta^k_n}{\sum}\prod_{j=1}^{n}\frac{u_j^{k_j}(t)}{k_j!}\,\mathrm{d}t,\\
		&=&\frac{\alpha}{e^{\beta (n+1)}} \int_0^{x}t^n\,\mathrm{d}t\sum_{k=1}^{n}\left(-\beta\right)^k\underset{\Theta^k_n}{\sum}\prod_{j=1}^{n}\frac{1}{k_j!}\left(\frac{(-1)^{j+1}\alpha^j\beta^{j-1}}{j}\right)^{k_j}\\
		&=&\frac{\alpha x^{n+1}}{(n+1)e^{\beta (n+1)}}\sum_{k=1}^{n}\left(-\beta\right)^k\underset{\Lambda^k_n}{\sum}\prod_{j=1}^{n-k+1}\frac{1}{k_j!}\left(\frac{(-1)^{j+1}(j-1)!\alpha^j\beta^{j-1}}{j!}\right)^{k_j},
	\end{eqnarray*}
	where the last two steps follow from (\ref{2.10kk}) and Lemma \ref{l4}, respectively. Finally by using (\ref{2.10kk}) and rearranging the terms, we get
	\begin{equation*}
	\sum_{k=1}^{n}\left(-\beta\right)^k\underset{\Lambda^k_n}{\sum}\prod_{j=1}^{n-k+1}\frac{n!}{k_j!}\left(\frac{(-1)^{j+1}(j-1)!\alpha^j\beta^{j-1}}{j!}\right)^{k_j}=n!(-\alpha\beta)^n,
	\end{equation*}
	which is the required identity.
\end{proof}
Substituting $(\alpha,\beta)\in\{(-1,1),(1,1),(-1,-1),(1,-1)\}$ in (\ref{2.4kk}), we obtain the following results, respectively.
\begin{corollary}\label{ct2kk}
	For any positive integer $n$, the following identities hold:
	\begin{eqnarray}\label{cc2.4kk}
	\sum_{k=1}^{n‎}\left(-1\right)^kB_{n,k}(-0!,-1!,\ldots,-(n-k)!)&=&n!,\nonumber\\
	\sum_{k=1}^{n‎}\left(-1\right)^kB_{n,k}(0!,-1!,\ldots,(-1)^{n-k}(n-k)!)&=&n!(-1)^n,\nonumber\\
	\sum_{k=1}^{n‎}B_{n,k}(-0!,1!,\ldots,(-1)^{n-k+1}(n-k)!)&=&n!(-1)^n,\label{bg2.4kk}\\
	\sum_{k=1}^{n‎}B_{n,k}(0!,1!,\ldots,(n-k)!)&=&n!.\label{gb2.4kk}
	\end{eqnarray}
\end{corollary}
\noindent {\small{\sc Remark 2.}}
The identities given by (\ref{bg2.4kk}) and (\ref{gb2.4kk}) above can be expressed in terms of the $n$-th complete exponential Bell polynomials as follows:
\begin{eqnarray*}
	B_{n}(-0!,1!,\ldots,(-1)^{n}(n-1)!)&=&n!(-1)^n,\\
	B_{n}(0!,1!,\ldots,(n-1)!)&=&n!.
\end{eqnarray*}
Note that (\ref{gb2.4kk}) is a known result, \textit{i.e.} $B_{n,k}(0!,1!,\ldots,(n-k)!)=s(n,k),$ where $s(n,k)$'s are Stirling numbers of the first kind.

The corresponding identities for the partial ordinary Bell polynomials easily follow from (\ref{ew1.10}). 
\begin{corollary}\label{tokk}
	Let $\alpha,\beta$ be any real numbers and $n$ be any positive integer. Then
	\begin{equation}\label{mdh}
	\sum_{k=1}^{n‎}\frac{\left(-\beta\right)^k}{k!}\hat{B}_{n,k}\left(\alpha,-\frac{\alpha^2\beta}{2},\ldots,\left(-1\right)^{n-k}\frac{\alpha^{n-k+1}\beta^{n-k}}{n-k+1}\right)=(-\alpha\beta)^n.
	\end{equation}
\end{corollary}
Substituting $(\alpha,\beta)\in\{(-1,1),(1,1),(-1,-1),(1,-1)\}$ in (\ref{mdh}), we obtain the following results, respectively.
\begin{corollary}\label{cot2kk}
	For any positive integer $n$, the following identities hold:
	\begin{eqnarray*}\label{coc2.4kk}
		\sum_{k=1}^{n‎}\frac{\left(-1\right)^k}{k!}\hat{B}_{n,k}\left(-1,-\frac{1}{2},\ldots,-\frac{1}{n-k+1}\right)&=&1,\\
		\sum_{k=1}^{n‎}\frac{\left(-1\right)^k}{k!}\hat{B}_{n,k}\left(1,-\frac{1}{2},\ldots,\frac{(-1)^{n-k}}{n-k+1}\right)&=&(-1)^n,\\
		\sum_{k=1}^{n‎}\frac{1}{k!}\hat{B}_{n,k}\left(-1,\frac{1}{2},\ldots,\frac{(-1)^{n-k+1}}{n-k+1}\right)&=&(-1)^n,\\
		\sum_{k=1}^{n‎}\frac{1}{k!}\hat{B}_{n,k}\left(1,\frac{1}{2},\ldots,\frac{1}{n-k+1}\right)&=&1.
	\end{eqnarray*}
\end{corollary}
Next we obtain an identity of the partial exponential Bell polynomials in terms of the falling factorials.
\begin{theorem}\label{t2}
	Let $\alpha$ be any non zero real number and $n$ be any positive integer. Then
	\begin{equation*}\label{2.4}
	\sum_{k=1}^{n‎}\left(1-\frac{1}{\alpha}\right)_kB_{n,k}((\alpha)_1,(\alpha)_2,\ldots,(\alpha)_{n-k+1})=(\alpha-1)_n,
	\end{equation*}
	where $(\alpha)_k=\alpha(\alpha-1)(\alpha-2)\ldots(\alpha-k+1)$ denotes the falling factorial.
\end{theorem}
\begin{proof}
	Consider the following ordinary differential equation
	\begin{equation}\label{2.5}
	\frac{du}{dx}=\alpha u^{1-1/\alpha},\ \ \ \ u(0)=1,\ \ \ |x|<1.
	\end{equation}
	Equivalently,
	\begin{equation}\label{2.6}
	u(x)=u(0)+\alpha \int_0^{x}u^{1-1/\alpha}(t)\,\mathrm{d}t.
	\end{equation}
	In (\ref{2.6}), there is no linear term but the nonlinear term is $N(u)=u^{1-1/\alpha}$. Substituting $u=\sum_{n=0}^{\infty}u_n$ and $N(u)=\sum_{n=0}^{\infty}A_n$ in the above equation and applying ADM, we get
	\begin{equation*}\label{2.7}
	\sum_{n=0}^{\infty}u_n(x)=u(0)+\alpha\sum_{n=0}^{\infty} \int_0^{x}A_n\left(u_0(t),u_1(t),\ldots,u_n(t)\right)\,\mathrm{d}t.
	\end{equation*}
	Therefore, for all $n\in\mathbb{N}$,
	\begin{equation}\label{2.8}
	u_{n+1}(x)=\alpha \int_0^{x}A_{n}\left(u_0(t),u_1(t),\ldots,u_n(t)\right)\,\mathrm{d}t.
	\end{equation}
	On comparing $u_0=u(0)=1$ and hence $u_1,u_2,u_3,u_4$ are recursively obtained as
	\begin{eqnarray*}
		A_0\left(u_0(x)\right)&=&u_0^{1-1/\alpha}=1,\\
		u_1(x)&=&(\alpha)_1 x,\\
		A_1\left(u_0(x),u_1(x)\right)&=&\left(1-\frac{1}{\alpha}\right)u_0^{-1/\alpha}u_1=\frac{(\alpha)_2}{\alpha} x,\\
		u_2(x)&=&\frac{(\alpha)_2}{2!} x^2,\\
		A_2\left(u_0(x),u_1(x),u_2(x)\right)&=&\frac{1}{2!}\left(1-\frac{1}{\alpha}\right)\left(-\frac{1}{\alpha}u_0^{-1/\alpha-1}u_1^2+2u_0^{-1/\alpha}u_2\right)\\
		&=&\frac{(\alpha)_3}{2!\alpha} x^2,\\
		u_3(x)&=&\frac{(\alpha)_3}{3!} x^3,\\
		A_3\left(u_0(x),u_1(x),u_2(x),u_3(x)\right)&=&\frac{1}{3!}\left(1-\frac{1}{\alpha}\right)\left(\left(\frac{1}{\alpha}+\frac{1}{\alpha^2}\right)u_0^{-1/\alpha-2}u_1^3\right.\\&&\left.-\frac{4}{\alpha}u_0^{-1/\alpha-1}u_1u_2+6u_0^{-1/\alpha}u_3-\frac{2}{\alpha}u_0^{-1/\alpha}u_1u_2\right)\\
		&=&\frac{(\alpha)_4}{3!\alpha} x^3,\\
		u_4(x)&=&\frac{(\alpha)_4}{4!} x^4,
	\end{eqnarray*}
	and so on.
	While using the separation of variables method, it is easy to see that $u(x)=(1+x)^{\alpha}$ is the solution of (\ref{2.5}). Since $|x|<1$, expand $(1+x)^{\alpha}$ by using generalized binomial theorem
	\begin{equation}\label{2.9}
	u(x)=1+\sum_{n=1}^{\infty}\frac{(\alpha)_n}{n!}x^n=u(0)+\sum_{n=1}^{\infty}u_n(x).
	\end{equation}
	By uniqueness of the solution of the differential equation (\ref{2.5}), the series solution obtained by ADM is consistent with (\ref{2.9}). Therefore, $n$-th component of $u(x)$ is
	\begin{equation}\label{2.10}
	u_{n}(x)=\frac{(\alpha)_{n}}{n!} x^{n},\ \ n\geq0.
	\end{equation}
	Now from (\ref{2.8}) and (\ref{2.10}), we have
	\begin{eqnarray*}\label{2.11}
		\frac{(\alpha)_{n+1}}{(n+1)!} x^{n+1}&=&\alpha \int_0^{x}A_{n}\left(u_0(t),u_1(t),\ldots,u_n(t)\right)\,\mathrm{d}t,\\
		&=&\alpha \int_0^{x}\sum_{k=1}^{n}N^{(k)}(u_0)\underset{\Theta^k_n}{\sum}\prod_{j=1}^{n}\frac{u_j^{k_j}(t)}{k_j!}\,\mathrm{d}t,\ \ \ \ \ \ \ \ \ \ \mathrm{(using\ (\ref{1.7})\ and\ (\ref{1.8}))}\\
		&=&\alpha \int_0^{x}t^n\,\mathrm{d}t\sum_{k=1}^{n}\left(1-\frac{1}{\alpha}\right)_k\underset{\Theta^k_n}{\sum}\prod_{j=1}^{n}\frac{1}{k_j!}\left(\frac{(\alpha)_{j}}{j!}\right)^{k_j},\ \ \ \mathrm{(using\ (\ref{2.10}))}\\
		&=&\frac{\alpha}{n+1}x^{n+1}\sum_{k=1}^{n}\left(1-\frac{1}{\alpha}\right)_k\underset{\Lambda^k_n}{\sum}\prod_{j=1}^{n-k+1}\frac{1}{k_j!}\left(\frac{(\alpha)_{j}}{j!}\right)^{k_j},
	\end{eqnarray*}
	where the last step follows on using Lemma \ref{l4}. On rearranging the terms we get
	\begin{equation*}
	\sum_{k=1}^{n}\left(1-\frac{1}{\alpha}\right)_k\underset{\Lambda^k_n}{\sum}n!\prod_{j=1}^{n-k+1}\frac{1}{k_j!}\left(\frac{(\alpha)_{j}}{j!}\right)^{k_j}=(\alpha-1)_n,
	\end{equation*}
	which is the required identity.
\end{proof}
\begin{corollary}\label{tcm2}
	Let $m,n$ be any positive integer such that $m>n$. Then
	\begin{equation*}\label{bn2.4}
	\sum_{k=1}^{n‎}\frac{1}{k!}\left(1-\frac{1}{m}\right)_k\hat{B}_{n,k}\left({{m}\choose{1}},{{m}\choose{2}},\ldots,{{m}\choose{n-k+1}}\right)={{m-1}\choose{n}},
	\end{equation*}
	where ${{m}\choose{n}}=\frac{m!}{n!(m-n)!}$.
\end{corollary}
Finally, we give two recursive algorithms for the complete exponential Bell polynomials.
\begin{theorem}
	Let $n$ be any positive integer. Then
	\begin{equation*}
	B_n(1!u_1,\ldots,n!u_n)=\sum_{k=0}^{n-1}(k+1)(n-1)_ku_{k+1}B_{n-k-1}(1!u_1,\ldots,(n-k-1)!u_{n-k-1}),
	\end{equation*}
	and
	\begin{align*}
	B_n(1!u_1,\ldots,n!u_n)&=u_{1}B_{n-1}(1!u_1,\ldots,(n-1)!u_{n-1})\\
	&\ \ \ +\sum_{k=1}^{n-1}(k+1)u_{k+1}\frac{\partial}{\partial u_k}B_{n-1}(1!u_1,\ldots,(n-1)!u_{n-1}).
	\end{align*}
\end{theorem}
\begin{proof}
	From Corollary 1 and 2 of {\small{\sc Duan}} (2011), we have the following recursive algorithms for  Adomian polynomials:
	\begin{eqnarray*}
		A_n(u_1,\ldots,u_n)&=&\frac{1}{n}\sum_{k=0}^{n-1}(k+1)u_{k+1}\frac{\partial}{\partial u_0}A_{n-k-1}(u_1,\ldots,u_{n-k-1}),\\
		A_n(u_1,\ldots,u_n)&=&\frac{1}{n}\sum_{k=0}^{n-1}(k+1)u_{k+1}\frac{\partial}{\partial u_k}A_{n-1}(u_1,\ldots,u_{n-1}),\ \ \ n\geq 1.
	\end{eqnarray*}
	The proof follows by choosing the nonlinear term $N(u)=e^u$ in Theorem \ref{t1} and using (\ref{2.1we}) in the above expressions for Adomian polynomials.
\end{proof}
Next corollary gives the corresponding results for the complete ordinary Bell polynomials.
\begin{corollary}
	For a positive integer $n$, the following holds:
	\begin{equation*}
	\hat{B}_n(u_1,u_2,\ldots,u_n)=\frac{1}{n}\sum_{k=0}^{n-1}(k+1)u_{k+1}\hat{B}_{n-k-1}(u_1,u_2,\ldots,u_{n-k-1}),
	\end{equation*}
	and
	\begin{align*}
	\hat{B}_n(u_1,u_2,\ldots,u_n)&=\frac{u_{1}}{n}\hat{B}_{n-1}(u_1,u_2,\ldots,u_{n-1})\\
	&\ \ \ +\frac{1}{n}\sum_{k=1}^{n-1}(k+1)u_{k+1}\frac{\partial}{\partial u_k}B_{n-1}(u_1,u_2,\ldots,u_{n-1}),
	\end{align*}
	where
	\begin{equation*}
	\hat{B}_n(u_1,u_2,\ldots,u_n)=\sum_{k=1}^{n}\frac{1}{k!}\hat{B}_{n,k}(u_1,u_2,\ldots,u_{n-k+1}).
	\end{equation*}
\end{corollary}
\section*{Acknowledgements}
The authors wish to thank Dr. Randolph Rach for providing insightful comments on initial version of the manuscript and especially for pointing reference {\small{\sc Abbaoui}} {\it et al.} (1995). The research of the first author was supported by a UGC fellowship no. F.2-2/98(SA-I), Govt. of India.

\end{document}